\documentclass[11pt,twoside]{amsart}
\usepackage[latin1]{inputenc}
\usepackage{amsmath}
\usepackage{amsthm}
\usepackage{amssymb}
\usepackage[all]{xy}
\usepackage{hyperref}
\newtheorem{thm}{Theorem}[section]

\newtheorem{prop}[thm]{Proposition}

\newtheorem{lem}[thm]{Lemma}
\newtheorem{cor}[thm]{Corollary}

\numberwithin{equation}{section}

\theoremstyle{definition}
\newtheorem{definition}[thm]{Definition}
\newtheorem{remark}[thm]{Remark}
\newtheorem{ex}[thm]{Example}

\newcommand{\Db}{{\rm D}^{\rm b}}

\newcommand{\Aut}{{\rm Aut}}

\newcommand{\rk}{{\rm rk}}

\newcommand{\Ext}{{\rm Ext}}

\newcommand{\cal}{\mathcal}
\newcommand{\ka}{{\cal A}}

\newcommand{\kc}{{\cal C}}

\newcommand{\kp}{{\cal P}}

\newcommand{\ZZ}{\mathbb{Z}}
\newcommand{\QQ}{\mathbb{Q}}
\newcommand{\RR}{\mathbb{R}}
\newcommand{\CC}{\mathbb{C}}

\newcommand{\HH}{\mathbb{H}}
\newcommand{\PP}{\mathbb{P}}

\newcommand{\Co}{{C\!o}}

\newcommand{\OO}{{\rm O}}
\newcommand{\wL}{\Lambda}
\renewcommand{\to}{\xymatrix@1@=15pt{\ar[r]&}}
\renewcommand{\rightarrow}{\xymatrix@1@=15pt{\ar[r]&}}
\renewcommand{\mapsto}{\xymatrix@1@=15pt{\ar@{|->}[r]&}}
\renewcommand{\twoheadrightarrow}{\xymatrix@1@=18pt{\ar@{->>}[r]&}}
\renewcommand{\hookrightarrow}{\xymatrix@1@=15pt{\ar@{^(->}[r]&}}
\newcommand{\hook}{\xymatrix@1@=15pt{\ar@{^(->}[r]&}}
\newcommand{\congpf}{\xymatrix@1@=15pt{\ar[r]^-\sim&}}
\renewcommand{\cong}{\simeq}

\begin{document}

\title{On derived categories of K3 surfaces, symplectic automorphisms and the Conway group}

\author[D.\ Huybrechts]{Daniel Huybrechts}

\address{Mathematisches Institut,
Universit{\"a}t Bonn, Beringstr.\ 1, 53115 Bonn, Germany}
\email{huybrech@math.uni-bonn.de}
\begin{abstract} \noindent
  In this note we interpret a recent result of Gaberdiel et al \cite{Gaberdiel} in terms of derived equivalences
of K3 surfaces.  We prove that there is a natural bijection between subgroups of the Conway group
$\Co_0$ with invariant lattice of rank at least four and groups of symplectic derived equivalences of $\Db(X)$
of projective K3 surfaces fixing a stability condition.

As an application we prove that every such subgroup $G\subset\Co_0$ satisfying an additional
condition can be realized as a group of symplectic automorphisms of an irreducible
symplectic variety deformation equivalent to ${\rm Hilb}^n(X)$ of some K3 surface.
 \vspace{-2mm}
\end{abstract}
\dedicatory{Dedicated to Professor Shigeru Mukai on the occasion of his 60th birthday.}

\maketitle
{\let\thefootnote\relax\footnotetext{This work was supported by the SFB/TR 45 `Periods,
Moduli Spaces and Arithmetic of Algebraic Varieties' of the DFG
(German Research Foundation).}
}

\maketitle

%
%


In his celebrated paper \cite{MukMat} Mukai established a bijection between
finite groups of symplectic automorphisms of K3 surfaces $G\subset\Aut_s(X)$
and finite subgroups $G\subset M_{23}$ of the Mathieu group $M_{23}$ with at least five orbits.
An alternative approach relying on Niemeier lattices was  given by Kond{\=o} in \cite{Kondo}.

More recently, physicists observed that groups of supersymmetry preserving automorphisms 
of non-linear $\sigma$-models on K3 surfaces are linked to subgroups of the larger Mathieu group
$M_{24}$ and the even larger Conway group ${C\!o}_1$, both sporadic finite simple groups. The lattice theory used in \cite{Gaberdiel},
ultimately going back to Kond{\=o}, can be reinterpreted in purely mathematical terms 
to prove the following result about derived autoequivalences of K3 surfaces which should be
seen as a derived version of Mukai's classical result.

\begin{thm}\label{thm:main}
For a group $G$ the following conditions are equivalent:\\
i) $G$  is isomorphic to a subgroup of the group $\Aut_s(\Db(X),\sigma)$ of
some complex projective K3 surface $X$ endowed with a
stability condition $\sigma\in{\rm Stab}^{\rm o}(X)$.\\
ii) $G$ is isomorphic to a subgroup of
the Conway group ${C\!o}_0$ with invariant lattice of rank at least four.
\end{thm}

Here, $\Db(X)=\Db({\rm Coh}(X))$ is the bounded derived category of coherent sheaves on $X$ and $${\rm Stab}^{\rm o}(X)\subset
{\rm Stab}(X)$$ is Bridgeland's distinguished connected component of the space of stability conditions on $\Db(X)$,
see \cite{BrK3}. By $\Aut(\Db(X))$ we denote the group of isomorphism classes
of $\CC$-linear exact auto\-equivalences of the triangulated category $\Db(X)$ 
and by $\Aut_s(\Db(X))\subset\Aut(\Db(X))$  the finite index subgroup of symplectic autoequivalences
(see below for details of these definitions).
Finally, we write $\Aut(\Db(X),\sigma)\subset\Aut(\Db(X))$ for the subgroup of autoequivalences $\Phi$ with
$\Phi^*\sigma=\sigma$ and let $$\Aut_s(\Db(X),\sigma):=\Aut(\Db(X),\sigma)\cap \Aut_s(\Db(X)).$$

To explain the condition in ii), recall that the Conway group $\Co_0$ is by definition the
ortho\-gonal group  of the Leech lattice $N$, i.e.
$$\Co_0:=\OO(N).$$
So, for a subgroup $G\subset\Co_0$ we can consider the invariant lattice $N^G$ and
 ii) means $\rk\,N^G\geq4$. 
Note that whenever $N^G$ is non-trivial, then $G$ does
not contain $-{\rm id}$ and can therefore be realized as a subgroup of the Conway
group $\Co_1:=\Co_0/\{\pm{\rm id}\}$.
We think of the condition $\rk\, N^G\geq 4$ as $G$ acting with at least four
orbits, analogously to Mukai's condition on subgroups $G\subset M_{23}$ acting with at least five
orbits on $\{1,\ldots,24\}$.

\medskip

Note that a finite group $G\subset \Aut(X)$ always leaves invariant one ample class $\alpha\in H^{1,1}(X,\ZZ)$
and, therefore, can be lifted to a subgroup of $\Aut(\Db(X),\sigma_\alpha)$, where
$\sigma_\alpha$ is the canonical stability condition with stability function $Z=\langle\exp(i\alpha),~~\rangle$ constructed
in \cite[Sec.\ 6\,\&\,7]{BrK3}. Also, as we shall see, the group $\Aut(\Db(X),\sigma)$ is automatically finite
for any $\sigma\in{\rm Stab}^{\rm o}(X)$, so that all groups in i) (and of course also in ii)) are finite.
Thus, Theorem \ref{thm:main} can indeed be seen as a true generalization of
Mukai's result on finite groups of symplectic automorphisms \cite{MukMat}.

\medskip

Furthermore, Theorem \ref{thm:main} can be used to prove that most of the above groups can be realized
as symplectic automorphisms on higher dimensional analogues of K3 surfaces.


\begin{thm}\label{thm:main2} Assume $G\subset\Co_0$ is a subgroup with
invariant lattice of rank at least four satisfying condition ($\ast$) (see Section \ref{sec:Aut}). 
Then there exists a projective irreducible symplectic variety $Y$ deformation equivalent to
the Hilbert scheme ${\rm Hilb}^n(X)$ of subschemes of length $n$ on a K3 surface $X$ such that 
$G$ is isomorphic to a subgroup of ${\rm Aut}_s(Y)$.
\end{thm}

A more complete result concerning groups of symplectic automorphisms
of deformations of ${\rm Hilb}^n(X)$ has recently been announced by Giovanni Mongardi, see
Remark \ref{rem:Mong}.

\bigskip

\noindent
{\bf Outline.}  Following \cite{AM},  a mathematician should think of  a non-linear $\sigma$-model on a K3 surface as a pair
of orthogonal positive planes $$P_1\perp P_2\subset\Lambda\otimes\RR$$ without any integral
$(-2)$-classes in $(P_1\oplus P_2)^\perp$. Here, $\Lambda:=E_8(-1)^{\oplus 2}\oplus U^{\oplus 4}$  
is the unique even, unimodular lattice of signature $(4,20)$. In fact, $\Lambda$ is isomorphic
to the Mukai lattice $\widetilde H(X,\ZZ)$ (which is nothing but $H^*(X,\ZZ)$
with the sign reversed in the pairing of $H^0$ and $H^4$) of any K3 surface $X$. 

Geometrically this comes about as follows. To any K3 surface $X$ with a K\"ahler class $\alpha\in H^{1,1}(X,\RR)$
one can naturally associate  two positive planes  $$P_X:=(H^{2,0}\oplus H^{0,2})(X)\cap H^2(X,\RR)
\text{ and }P_\alpha:=\RR\cdot\alpha\oplus \RR\cdot(1-(\alpha)^2/2)$$ in  $\widetilde H(X,\RR)$.
A $(-2)$-class $\delta\in \widetilde H(X,\ZZ)$ is orthogonal to $P_X$ if and only if $\delta$ is algebraic,
i.e.\ $\delta\in\widetilde H^{1,1}(X,\ZZ)=H^0(X,\ZZ)\oplus H^{1,1}(X,\ZZ)\oplus H^4(X,\ZZ)$.
If $\delta\in H^{1,1}(X,\ZZ)$, then $(\alpha.\delta)\ne0$ and, therefore, $\delta\not\in P_\alpha^\perp$. In fact
it can be shown that for $\alpha\in H^{1,1}(X,\ZZ)\otimes\RR$, e.g.\ $\alpha$ an ample
class, with $(\alpha)^2>2$ none of the algebraic $(-2)$-classes is orthogonal to $P_\alpha$.

Note that the positive four space $P_1\oplus P_2$ associated with  a pair of orthogonal positive planes $P_1,P_2\subset H^*(M,\RR)$
can  always be written as $\exp(B)\cdot(P_X\oplus P_\alpha)$ for some K3 surface structure $X$  on the underlying
differentiable manifold $M$ endowed with a class
$\alpha\in H^{1,1}(X,\RR)$ with $(\alpha)^2>0$ and a class $B\in H^2(X,\RR)$. Here, $\exp(B)=1+B+(B)^2/2\in H^*(X,\RR)$
acts by multiplication. See \cite[Prop.\ 3.6]{HuyCY}.

\smallskip

The main result in \cite{Gaberdiel} describes all finite subgroups $G\subset \OO(\Lambda)$
 acting trivially on $P_1\oplus P_2$ for some non-linear $\sigma$-model $P_1\perp P_2$ as above.
So the main task of this note is to  pass from this lattice theoretic condition on  finite groups 
of isometries to one that can be phrased in terms of derived categories $\Db(X)$ of complex projective
K3 surfaces. In fact \cite{Gaberdiel} also contains a more precise description of the occurring groups
which of course includes all groups of Mukai's list (as any finite group of automorphisms preserves one ample
class). But it also contains groups
of the form $(\ZZ/3\ZZ)^{\oplus 4}.A_6$, which in particular is not contained in $M_{24}$
as its order does not divide $|M_{24}|$.

\medskip

In Section \ref{sec:Hodge} we establish a bijection between groups of autoequivalences fixing a
stability condition $\sigma\in{\rm Stab}^{\rm o}(X)$ and groups of Hodge isometries of $\widetilde H(X,\ZZ)$
acting trivially on the positive four space that comes with $\sigma$.
In Section \ref{sec:Gab} we give a brief sketch of the proof of \cite{Gaberdiel}.
The final Section \ref{sec:Proof} contains the proof of Theorem \ref{thm:main}.

\medskip

\noindent
{\bf Acknowledgements.} I wish to thank the organizers of the conference `Development of Moduli Theory' in Kyoto
in June 2013 and in particular Shigeyuki Kond{\=o} for organizing the memorable event on the
occasion of the 60th birthday of Shigeru Mukai whose influence on the theory of K3 surfaces
in general and on the topics related to this article can hardly be overestimated.

I am grateful to Roberto Volpato and  Tom Bridgeland for insightful comments on preliminary versions
of this paper, to Giovanni Mongardi for the question prompting Theorem \ref{thm:main2} and to the referee for pertinent questions.

\section{Lifting Hodge isometries}\label{sec:Hodge}

We shall link Hodge isometries of the Mukai lattice $\widetilde H(X,\ZZ)$ of a complex
projective K3 surface $X$ fixing an additional positive plane in $\widetilde H^{1,1}(X,\ZZ)\otimes\RR$
to autoequivalences of $\Db(X)$ fixing a stability condition.

\subsection{}\label{sec:abstrsett}
Let $\wL$ be  the lattice $E_8(-1)^{\oplus 2}\oplus U^{\oplus 4}$ (or any
lattice of signature $(4,n)$) and consider a K3 Hodge structure on $\wL$, i.e.\
a Hodge structure of weight two given by an orthogonal decomposition
$$\wL\otimes\CC=\wL^{2,0}\oplus\wL^{1,1}\oplus\wL^{0,2}$$
such that $\Lambda^{2,0}$ is isotropic and
$(\wL^{2,0}\oplus\wL^{0,2})\cap(\wL\otimes\RR)\subset\wL\otimes\RR$ is a positive
plane.

A \emph{Hodge isometry} $\varphi:\wL\congpf\wL$ is an orthogonal transformation $\varphi\in\OO(\wL)$ such that its $\CC$-linear extension
satisfies $\varphi(\wL^{2,0})=\wL^{2,0}$. We say that $\varphi$ is \emph{symplectic} if $\varphi|_{\wL^{2,0}}={\rm id}$ and
\emph{positive} if there exists a positive plane $P\subset \wL^{1,1}\cap(\wL\otimes\RR)$ with
$\varphi|_P={\rm id}$. If $P$ is given and $\varphi|_P={\rm id}$, then $\varphi$ is called
\emph{$P$-positive}.

\subsection{}
Let now $\widetilde H(X,\ZZ)$ be the Mukai lattice of a complex K3 surface $X$ with its natural Hodge structure
given by $\widetilde H^{2,0}(X)=H^{2,0}(X)$ and $\widetilde H^{1,1}(X)=(H^0\oplus H^{1,1}\oplus H^4)(X)$. The group of all Hodge isometries
resp.\ of all symplectic Hodge isometries
 shall be denoted
$$\Aut(\widetilde H(X,\ZZ))\text{ resp. }\Aut_s(\widetilde H(X,\ZZ)).$$

For a positive plane $P\subset \widetilde H^{1,1}(X,\RR)$ we write
$$\Aut(\widetilde H(X,\ZZ),P)\subset \Aut(\widetilde H(X,\ZZ))$$ for
the subgroup of all $P$-positive Hodge isometries and $\Aut_s(\widetilde H(X,\ZZ),P)$ for its intersection
with $\Aut_s(\widetilde H(X,\ZZ))$. If $P=P_Z$ is a positive plane spanned by
${\rm Re}(Z)$ and ${\rm Im}(Z)$ of some $Z\in\widetilde H^{1,1}(X,\ZZ)\otimes\CC$, then let
$$\Aut(\widetilde H(X,\ZZ),Z):=\Aut(\widetilde H(X,\ZZ),P_Z)\subset \Aut(\widetilde H(X,\ZZ)).$$

\begin{ex}
Any K\"ahler class $\alpha\in H^{1,1}(X)$ gives rise to a po\-si\-tive plane
$P_\alpha\subset \widetilde H^{1,1}(X,\RR)$ spanned by $\alpha\in H^{1,1}(X)$ and $1-(\alpha)^2/2\in
(H^0\oplus H^4)(X,\RR)$. We write $$\Aut(\widetilde H(X,\ZZ),\alpha):=\Aut(\widetilde H(X,\ZZ),P_\alpha).$$
If  $Z=\exp(i\alpha)=1+i\alpha-(\alpha)^2/2$, then $P_\alpha=P_Z$ and
$\Aut(\widetilde H(X,\ZZ),\alpha)=\Aut(\widetilde H(X,\ZZ),Z)$.
The setting can be generalized to the positive plane generated by real and imaginary part of
$\exp(B+i\alpha)\in\widetilde H^{1,1}(X)$, where $\alpha\in H^{1,1}(X,\RR)$ with $(\alpha)^2>0$ 
and $B\in H^{1,1}(X,\RR)$. 
\end{ex}

Let $X$ be a complex K3 surface  and $f:X\congpf X$ an automorphism. If $f$ acts as ${\rm id}$ on
$H^{2,0}(X)$, then $f^*:\widetilde H(X,\ZZ)\congpf \widetilde H(X,\ZZ)$ is a symplectic Hodge isometry, i.e.\ 
$f^*\in\Aut_s(\widetilde H(X,\ZZ))$.
If $f^*\alpha=\alpha$ for a K\"ahler class $\alpha$ or, more generally, for some $\alpha\in H^{1,1}(X,\RR)$ with $(\alpha)^2>0$, then $f^*$ is
$P_\alpha$-positive, i.e.\
$f^*\in\Aut(\widetilde H(X,\ZZ),\alpha)$.

If an automorphism $f:X\congpf X$ is of finite order, then $f^*$ is always positive.
Indeed, take any K\"ahler class $\alpha\in H^{1,1}(X,\RR)$ and consider $\widetilde\alpha:=\sum_if^{i*}\alpha$. Then $\widetilde\alpha$ is  a K\"ahler class that satisfies $f^*(\widetilde\alpha)=\widetilde\alpha$ 
and hence $f^*={\rm id}$ on $P_{\widetilde\alpha}$.
So any symplectic automorphism of finite order $f:X\congpf X$ of a K3 surface $X$ gives rise to
a  positive symplectic Hodge isometry of $\widetilde H(X,\ZZ)$.

\begin{remark}\label{rem:finite}
i) As a sort of converse of the above, one observes that
$\Aut_s(\widetilde H(X,\ZZ),P)=\Aut_s(\widetilde H(X,\ZZ))\cap\Aut(\widetilde H(X,\ZZ),P)$ is a finite group:
Indeed, it is a discrete subgroup of the compact group
$$\OO(((H^{2,0}\oplus H^{0,2})(X,\RR)\oplus P)^\perp)\cong\OO(20,\RR).$$
In particular, any Hodge isometry which is symplectic and positive is automatic of finite order. In this sense,
Mukai's classification of finite groups of symplectic automorphisms of K3 surfaces is part of
the broader classification of all groups of symplectic $P$-positive Hodge isometries of $\widetilde H(X,\ZZ)$ for some
K3 surface $X$ endowed with a positive plane $P\subset \widetilde H^{1,1}(X,\RR)$.

ii) In fact, if $X$ is projective, then already $\Aut(\widetilde H(X,\ZZ),P)$ is finite. Indeed, in this case the transcendental
lattice $T(X)$ is non-degenerate and irreducible and hence $\Aut(T(X))$ is a discrete subgroup of a compact group and hence
finite. But the kernel of $\Aut(\widetilde H(X,\ZZ),P)\to\Aut(T(X))$ is contained in the finite group
$\Aut_s(\widetilde H(X,\ZZ),P)$. The finiteness of $\Aut(\widetilde H(X,\ZZ),P)$ is the analogue of the
finiteness of the group of all automorphisms $f:X\congpf X$ fixing an ample line bundle $L$ or a K\"ahler class $\alpha$.
\end{remark}
\subsection{}
Let from now on the complex K3 surface  $X$ also be  projective.  We denote its bounded derived category of coherent sheaves by
$\Db(X):=\Db({\rm Coh}(X))$. Furthermore, let $\Aut(\Db(X))$ be
the group of isomorphism classes of exact $\CC$-linear autoequivalences $\Phi:\Db(X)\congpf\Db(X)$.
To any $\Phi\in\Aut(\Db(X))$ one associates the Hodge isometry
$\varphi:=\Phi^H:\widetilde H(X,\ZZ)\congpf\widetilde H(X,\ZZ)$ which defines
a homomorphism
$$\rho:\Aut(\Db(X))\to\Aut(\widetilde H(X,\ZZ)).$$ This goes back to Mukai's article \cite{MukaiVB}, see  \cite[Ch.\ 10]{HuyFM} for further
details, references, and notations.
The image of $\rho$ is the index two subgroup
$$\Aut(\widetilde H(X,\ZZ))^+\subset\Aut(\widetilde H(X,\ZZ))$$ of Hodge isometries  preserving the (natural)
orientation of  positive  planes $P\subset\widetilde H^{1,1}(X,\RR)$, see \cite{HMS}.
We say that $\Phi$ is symplectic if $\varphi\in \Aut_s(\widetilde H(X,\ZZ))$ and we let
$$\Aut_s(\Db(X))\subset\Aut(\Db(X))$$
denote the subgroup of all symplectic autoequivalences.


In the following we shall denote by ${\rm Stab}(X)$ the space of all stability conditions $\sigma=(\kp,Z)$ on
$\Db(X)$ and by ${\rm Stab}^{\rm o}(X)\subset{\rm Stab}(X)$ the distinguished connected component introduced and studied in 
\cite{BrK3}. Tacitly, all stability conditions are required to be locally finite. For a brief survey of the main features see also \cite{HuyNewton}.

The group $\Aut(\Db(X))$ acts on ${\rm Stab}(X)$ and we let 
$$\Aut^{\rm o}(\Db(X))\subset\Aut(\Db(X))$$ be
the subgroup of all autoequivalences fixing ${\rm Stab}^{\rm o}(X)$. 
Note that it is known that the restriction $\rho:\Aut^{\rm o}(\Db(X))\twoheadrightarrow \Aut(\widetilde H(X,\ZZ))^+$ is still surjective.
Indeed, in the original argument, due to Mukai and Orlov see \cite[Ch.\ 10.2]{HuyFM}, one only has to check that universal families of
$\mu$-stable bundles preserve the distinguished component ${\rm Stab}^{\rm o}\subset{\rm Stab}$
and for this see e.g.\ \cite[Prop.\ 5.2]{HuyJAG}.

Conjecturally,
${\rm Stab}^{\rm o}(X)={\rm Stab}(X)$ or, at least, $\Aut^{\rm o}(\Db(X))=\Aut(\Db(X))$.
In fact, a proof of the conjecture for the case $\rho(X)=1$ has recently been given by
Bayer and Bridgeland \cite{BB}.
In any case, as shown in \cite{BrK3}, the group $\Aut^{\rm o}(\Db(X))\cap{\rm Ker}(\rho)$ can be identified
with the group of deck transformations of the covering
$$\pi:{\rm Stab}^{\rm o}(X)\twoheadrightarrow\kp^+_0(X),~~\sigma=(\kp,Z)\mapsto Z.$$
Here, $$\kp^+_0(X)=\kp^+(X)\setminus\bigcup_{\delta\in\Delta_X}\delta^\perp\subset 
\widetilde H^{1,1}(X,\ZZ)\otimes\CC$$ with $\kp^+(X)$ the connected component containing
$1+i\alpha-(\alpha)^2/2$ with $\alpha$ ample of the open
set $\kp(X)$ of all $Z\in\widetilde H^{1,1}(X,\ZZ)\otimes\CC$ with real and imaginary part spanning a positive plane $$P_Z:=\RR\cdot{\rm Re}(Z)\oplus\RR\cdot{\rm Im}(Z)\subset\widetilde H^{1,1}(X,\RR).$$ By $\Delta_X\subset \widetilde H^{1,1}(X,\ZZ)$ we denote the set of
all  $(-2)$-classes $\delta\in\widetilde H^{1,1}(X,\ZZ)$. Here and in the sequel, the stability function
$Z:\widetilde H^{1,1}(X,\ZZ)\to\CC$ is, via
Poincar\'e duality, identified with an element $Z\in \widetilde H^{1,1}(X,\ZZ)\otimes\CC$.


\begin{definition} An exact $\CC$-linear autoequivalence
$\Phi:\Db(X)\congpf\Db(X)$ is \emph{positive} if there exists a
stability condition $\sigma=(\kp,Z)\in{\rm Stab}^{\rm o}(X)$ with $\Phi^*\sigma=\sigma$
(and then $\Phi$ is called \emph{$\sigma$-positive}). The group of all $\sigma$-positive
autoequivalences is denoted $\Aut(\Db(X),\sigma)$.
\end{definition}

Note that in particular $\Aut(\Db(X),\sigma)\subset\Aut^{\rm o}(\Db(X))$.
\subsection{} The next proposition is a derived version 
of the Global Torelli theorem for automorphisms of polarized K3 surfaces $(X,L)$
which can be stated as
$$\Aut(X,L)\congpf \Aut(H^2(X,\ZZ),\ell)$$
for $\Aut(X,L)$ the group of automorphisms $f:X\congpf X$ with
$f^*L\cong L$  and $\Aut(H^2(X,\ZZ),\ell)$ the group of Hodge isometries
of $H^2(X,\ZZ)$ fixing the ample class $\ell:={\rm c}_1(L)$.

\begin{prop}\label{prop:liftHodge}
For $\sigma=(\kp,Z)\in{\rm Stab}^{\rm o}(X)$, the homomorphism $\rho$
  induces
isomorphisms
$$\Aut(\Db(X),\sigma)\congpf\Aut(\widetilde H(X,\ZZ),Z)$$
and
$$\Aut_s(\Db(X),\sigma)\congpf\Aut_s(\widetilde H(X,\ZZ),Z).$$
\end{prop}

\begin{proof}
Clearly, for $\Phi\in\Aut(\Db(X),\sigma)$ the induced Hodge isometry $\varphi:=\rho(\Phi)$ satisfies $\varphi(Z)=Z$ and
hence $\varphi\in \Aut(\widetilde H(X,\ZZ),Z)$.

To prove injectivity, let $\Phi\in\Aut(\Db(X),\sigma)$ with $\varphi={\rm id}$. Then by
\cite{BrK3}, $\Phi$ acts as a deck transformation of the covering map $\pi:{\rm Stab}^{\rm o}(X)\twoheadrightarrow\kp_0^+(X)$.
But a deck transformation that fixes a point $\sigma\in{\rm Stab}^{\rm o}(X)$ has to be the identity. Hence, $\Phi\cong{\rm id}$.

For the surjectivity, let $\varphi\in \Aut(\widetilde H(X,\ZZ),Z)$. As
$\Aut(\widetilde H(X,\ZZ),Z)\subset\Aut(\widetilde H(X,\ZZ))^+$, there exists an autoequivalence
$\Phi_0\in\Aut^{\rm o}(\Db(X))$ with $\rho(\Phi_0)=\varphi$.
Now $\sigma,\sigma_0:=\Phi_0^*\sigma\in{\rm Stab}^{\rm o}(X)$ both map to
$Z=\pi(\sigma)=\pi(\sigma_0)\in\kp_0^+(X)$ and, therefore, differ by a unique $\Psi\in{\rm Aut}^{\rm o}(\Db(X))$ with
$\rho(\Psi)={\rm id}$. But then $\Phi:=\Phi_0\circ\Psi\in{\rm Aut}^{\rm o}(\Db(X))$ satisfies
$\rho(\Phi)=\rho(\Phi_0)=\varphi$ and $\Phi^*\sigma=\sigma$, i.e.\ $\Phi\in\Aut(\Db(X),\sigma)$.

The second isomorphism  
follows  from the first.
\end{proof}

By Remark \ref{rem:finite} the proposition immediately yields the following. 
(Note that whenever $\Db(X)$  is used the surface $X$ is assumed to be projective.)
\begin{cor}
The groups $\Aut(\Db(X),\sigma)$ and $\Aut_s(\Db(X),\sigma)$ are finite.\qed
\end{cor}

\begin{remark}
As was explained to me by Tom Bridgeland, the finiteness of the stabilizer $\Aut(\Db(X),\sigma)$  of a stability condition
$\sigma$
is a general phenomenon. Roughly, for any triangulated category ${\mathcal D}$ the
quotient $\Aut({\mathcal D},\sigma)/\Aut({\mathcal D},{\rm Stab}^{\rm o}({\mathcal D}))$ is finite. Here, $\Aut({\mathcal D},{\rm Stab}^{\rm o}({\mathcal D}))$ is the subgroup
of autoequivalences acting trivially on the connected component ${\rm Stab}^{\rm o}({\mathcal D})$ containing the stability condition
$\sigma$.
\end{remark}

As another consequence we find
\begin{cor} Let $\sigma=(\kp,Z)\in{\rm Stab}^{\rm o}(X)$.
Then for a group $G$ the following conditions are equivalent:\\
i) $G$ is isomorphic to a subgroup of $\Aut(\Db(X),\sigma)$
resp.\ $\Aut_s(\Db(X),\sigma)$.\\
ii) $G$ is isomorphic to a subgroup  of $\Aut(\widetilde H(X,\ZZ),Z)$
resp.\  $\Aut_s(\widetilde H(X,\ZZ),Z)$.\qed
\end{cor}

\begin{remark} 
i) The arguments show  more generally that for any subgroup $G\subset\Aut^{\rm o}(\Db(X))$
of positive autoequivalences the
restriction $\rho:G\to \Aut(\widetilde H(X,\ZZ))$ is injective.
Indeed, $H:={\rm Ker}(\rho:G\to\Aut(\widetilde H(X,\ZZ)))$ acts trivially on $\widetilde H(X,\ZZ)$ and, therefore,
consists of deck transformations of $\pi:{\rm Stab}^{\rm o}(X)\twoheadrightarrow\kp_0^+(X)$.
However, by assumption on $G$, there exists for any $\Phi\in H\subset G$  a stability condition
$\sigma\in{\rm Stab}^{\rm o}(X)$ with $\Phi^*\sigma=\sigma$.
 Therefore, $\Phi={\rm id}$, since non-trivial deck transformations act without fixed points.

ii) In a different direction, one can generalize  to groups of auto\-equi\-valences that fix a simply connected 
open set of stability functions. Assume $G\subset\Aut(\widetilde H(X,\ZZ))^+$ is a subgroup of autoequivalences
such that there exists a contractible open set $U\subset\kp_0^+(X)$ with $\varphi(U)=U$ for
all $\varphi\in G$. Then there exists a non-canonical group homomorphism lifting the inclusion 
$$\xymatrix{&\raisebox{1mm}{$G$}\ar@{^(->}[d]\ar@{-->}[dl]\\
\Aut^{\rm o}(\Db(X))\ar@{->>}[r]&\Aut(\widetilde H(X,\ZZ))^+.
}$$

Indeed, since $U$ is simply connected, any connected component 
$U'\subset\pi^{-1}(U)\subset{\rm Stab}^{\rm o}(X)$  of $\pi^{-1}(U)$ maps homeo\-morphically onto $U$. Pick one $U'\subset\pi^{-1}(U)$
and argue as above: For any $\varphi\in G$, there exists $\Phi_0\in\Aut^{\rm o}(\Db(X))$ with $\rho(\Phi_0)=\varphi$.
Pick $\sigma=(\kp,Z)\in U'$. Then $\pi(\Phi_0^*\sigma)\in U$ and thus there exists
a unique $\Psi\in{\rm Aut}^{\rm o}(\Db(X))$ with $\rho(\Psi)={\rm id}$ and $\Psi^*\Phi_0^*\sigma\in U'$. The new
$\Phi:=\Phi_0\circ\Psi\in{\rm Aut}^{\rm o}(\Db(X))$ satisfies
$\rho(\Phi)=\rho(\Phi_0)=\varphi$ and $\Phi^*(U')=(U')$. Moreover, $\Phi$ is unique with these
properties and we define the lift $G\to\Aut^{\rm o}(\Db(X))$ by $\varphi\mapsto\Phi$.
To see that this defines a group homomorphism consider $\varphi_1,\varphi_2\in G$
and let $\varphi_3:=\varphi_1\circ\varphi_2$.
For the (unique) lifts $\Phi_i$ of $\varphi_i$ with $\Phi_i^*(U')=U'$ the composition $\Psi:=\Phi_3^{-1}\circ(\Phi_1\circ\Phi_2)$ satisfies
$\rho(\Psi)={\rm id}$ and $\Psi^*(U')=U'$. Hence, $\Psi={\rm id}$ and, therefore, $\Phi_3=\Phi_1\circ\Phi_2$.
\end{remark}

\begin{remark}
In \cite{HuyCY} the notion of generalized K3 structures on the differentiable manifold $M$ underlying a 
K3 surface was introduced as an orthogonal pair of generalized Calabi--Yau structures $\varphi,\varphi'\in\ka^{2\ast}_\CC(M)$. The 
period of a generalized K3 structure was defined as the pair of positive planes $P_\varphi,P_{\varphi'}\subset \widetilde H(M,\RR)$ spanned
by real and imaginary parts of the cohomology classes $[\varphi]$ resp.\ $[\varphi']$. Isomorphisms
between generalized K3 structures include ${\rm Diff}(M)$ and $B$-field twists $\exp(B)$ by closed forms
$B\in\ka^2(M)$ with integral cohomology class $\beta:=[B]\in H^2(M,\ZZ)$.  Note that ${\rm Diff}(M)$ surjects
onto the index two subgroup $\OO(H^2(M,\ZZ))^+$ and that $\OO(\widetilde H(M,\ZZ))$ is generated
by $\OO(H^2(M,\ZZ))$, $\exp(\beta)$ with $\beta\in H^2(M,\ZZ)$, and $\OO((H^0\oplus H^4)(M,\ZZ))$.
(Only the latter has no interpretation on the level of forms.)
In this sense $\Aut_s(\widetilde H(X,\ZZ),Z)$ may be seen as the group of automorphisms
of the generalized K3 structure given by $\varphi_1=\omega$, a holomorphic two-form on $X$,
and $\varphi_2=\exp(B+i\alpha)\in\ka^{2\ast}_\CC(M)$ representing $Z$.
So Proposition \ref{prop:liftHodge} seems to suggests that automorphisms of $(\varphi_1,\varphi_2)$ can be interpreted
as automorphisms of a stability condition $\sigma$ on $\Db(X)$. It would be interesting to find a more direct approach
to this not relying on the period description of both sides.
\end{remark}
\section{Groups of Hodge isometries as subgroups of $\Co_1$}\label{sec:Gab}
For the convenience of the reader, we recall the lattice theoretic arguments in \cite{Gaberdiel}
which relate groups of positive symplectic Hodge isometries with subgroups of the Conway group  $\Co_1$.
Section \ref{sec:i-ii} is later used  in Section \ref{sec:Proof} to prove that groups of symplectic $\sigma$-positive autoequivalences
can be realized as subgroups of $\Co_1$, whereas the converse is based on Section \ref{sec:ii-i}. 

\subsection{}
We go back to the abstract setting of Section \ref{sec:abstrsett}
and let $\Lambda=E_8(-1)^{\oplus 2}\oplus U^{\oplus 4}$. We also
fix   a positive subspace  of dimension four $\Pi\subset\Lambda\otimes\RR$ such that
no $(-2)$-class $\delta\in\Lambda$ is contained in $\Pi^\perp$.

Then consider the subgroup
$$\Aut(\Lambda,\Pi)\subset\OO(\Lambda)$$
of all isometries $\varphi:\Lambda\congpf \Lambda$ such that its $\RR$-linear extension
satisfies $\varphi={\rm id}$ on $\Pi$. Thus, $\Aut(\Lambda,\Pi)\subset\OO(\Pi^\perp)$ and
since $\Pi^\perp$ is negative definite and hence $\OO(\Pi^\perp)$ compact, $\Aut(\Lambda,\Pi)$ is finite.
For a subgroup $G\subset\Aut(\Lambda,\Pi)$ we denote by
$$\Lambda^G\text{ and } \Lambda_G:=(\Lambda^G)^\perp$$
the invariant part resp.\ its orthogonal complement.
The group $G$ can be chosen arbitrary and one can even take $G=\Aut(\Lambda,\Pi)$.

Following the classical line of arguments, one first proves

\begin{lem}\label{lem:stand}
The lattice $\Lambda_G$ is negative definite of rank $\rk\,\Lambda_G\leq 20$ and does not contain any $(-2)$-classes.
The induced action of $G$ on its discriminant $A_{\Lambda_G}=\Lambda_G^*/\Lambda_G$ is trivial and
the minimal number of generators  of $A_{\Lambda_G}$ is bounded by
$\ell(A_{\Lambda_G})\leq 24-\rk\,\Lambda_G$.
\end{lem}

\begin{proof} 
As $\Lambda_G\subset\Pi^\perp$, the assumption on $\Pi$ implies that the lattice
$\Lambda_G$  does not contain any $(-2)$-class. Moreover, 
since $\Pi^\perp\subset\Lambda\otimes\RR$ is a negative definite
subspace of dimension $\leq 20$,
also $\Lambda_G$ is negative definite and of rank $\rk\,\Lambda_G\leq 20$. As
$\Lambda$ is unimodular and, as is easy to check, also $\Lambda^G$ is non-degenerate,  there exists an isomorphism $A_{\Lambda_G}\cong A_{\Lambda^G}$ which is compatible
with the induced action of $G$. 
Since the action on the latter is trivial, it is so on $A_{\Lambda_G}$.
The isomorphism also yields $\ell(A_{\Lambda_G})=\ell(A_{\Lambda^G})\leq\rk\,\Lambda^G=\rk\,\Lambda-\rk\,\Lambda_G$.
The arguments are quite standard, for more details see \cite{HK3} and references therein.
\end{proof}

The key idea in \cite{Kondo} is to embed $\Lambda_G$ (or rather $\Lambda_G\oplus A_1(-1)$) into
some Niemeier lattice, i.e.\ into one of the $24$ negative definite, even, unimodular lattices of rank $24$.
In our situation, Kond{\=o}'s approach is easy to adapt if the stronger inequality
\begin{equation}\label{ass:sup}
\ell(A_{\Lambda_G})<24-\rk\,\Lambda_G,
\end{equation} 
  is assumed. Indeed,
then by \cite[Thm.\ 1.12.2]{NikulinInt} there exists a primitive embedding
$$\Lambda_G\,\hookrightarrow N_i$$
into one of the Niemeier lattices $N_i$. Moreover, as $G$ acts trivially on $A_{\Lambda_G}$, its
action on $\Lambda_G$ can be extended to an action of $G$ on $N_i$ which is trivial
on $\Lambda_G^\perp\subset N_i$, see \cite[Thm.\ 1.6.1, Cor.\ 1.5.2]{NikulinInt}. In Kond{\=o}'s approach $\Lambda_G\oplus A_1(-1)$ is embedded
into a Niemeier lattice $N_i$. This excludes $N_i$ from being the Leech lattice $N$ which does not contain any $(-2)$-class.
So, under the additional assumption (\ref{ass:sup}), the group $G$ can be realized as a subgroup of $\OO(N_i)$
of a certain Niemeier lattice different from the Leech lattice. 
If $N_i$ is the Niemeier lattice with root lattice $A_1(-1)^{\oplus 24}$, this eventually
leads to an embedding $G\,\hookrightarrow M_{24}$.  Recall that in this case
$\OO(N_i)\cong M_{24}\ltimes (\ZZ/2\ZZ)^{\oplus 24}$.
The $(\ZZ/2\ZZ)^{\oplus 24}$ is avoided by $G$, as $\Lambda_G$ does not contain any $(-2)$-classes. Indeed, if $g(e_i)=-e_i$ for
some $g\in G$ and a root $e_i$, then $e_i$ would be orthogonal to $\Lambda^G$ and hence contained in $\Lambda_G$.
However, if $N_i$ is the Leech lattice $N$, then one only gets an embedding  into the much larger Conway group $\Co_0$
$$G\,\hookrightarrow \OO(N)=:\Co_0.$$
Note that in both cases, the invariant lattice $N_i^G$ satisfies $\rk\,N_i^G\geq4$.

By a clever twist of the argument, the authors of \cite{Gaberdiel} manage to prove
the existence of an embedding into the Leech lattice only assuming the weak inequality in (\ref{ass:sup}) which holds
due to Lemma \ref{lem:stand}.

\begin{prop}\label{prop:Gab}
{\bf (Gaberdiel, Hohenegger, Volpato)}
For a group $G$ the following conditions are equivalent:\\
i) $G$ is isomorphic to a subgroup of $\Aut(\Lambda,\Pi)$ for some
positive four space $\Pi\subset \Lambda\otimes\RR$ without $(-2)$-class
contained in $\Pi^\perp$. \\
ii) $G$ is isomorphic to a subgroup of the Conway group $\Co_0$
 with invariant lattice of rank
at least four.
\end{prop}

For completeness sake, we sketch the argument. Section \ref{sec:i-ii} shows that i) implies ii) and  Section \ref{sec:ii-i} deals with the converse. We follow the original \cite{Gaberdiel} closely.

\subsection{}\label{sec:i-ii} We shall show that there  exists a primitive embedding
$$\Lambda_G\,\hookrightarrow N$$ into the Leech lattice $N$ inducing
an inclusion $$G\,\hookrightarrow \Co_0$$
with $\rk\,N^G\geq4$.

If $\ell(A_{\Lambda_G})=24-\rk\,\Lambda_G$,
then an embedding into a Niemeier lattice can be found if for
odd prime $p$  the $p$-Sylow group $(A_{\Lambda_G})_p$ of $A_{\Lambda_G}$ satisfies
the stronger inequality $\ell((A_{\Lambda_G})_p)<24-\rk\,\Lambda_G$
and for $p=2$ the discriminant form $(A_{\Lambda_G},q)$ splits off  $(A_{A_1},q)$, see \cite[Thm 1.12.2]{NikulinInt}.
Of course, if $\Lambda_G$ splits off a summand $A_1(-1)$ both conditions are satisfied,
but in general it seems difficult to check.

Instead, in \cite{Gaberdiel} Nikulin's criterion is applied to $\Lambda_G':=\Lambda_G\oplus A_1(-1)$. This is of course
inspired by Kond{\=o}'s original argument, but unfortunately in the present situation one cannot hope for embeddings of
$\Lambda_G'$ into a Niemeier lattice. Instead, one obtains primitive
embeddings $$\Lambda_G\,\hookrightarrow \Lambda_G'\,\hookrightarrow \Gamma: =E_8(-1)^{\oplus 3}\oplus U.$$
Note that  $\Gamma$ is the unique even unimodular lattice of signature
$(1,25)$ which is often denoted ${\rm II}_{1,25}$. As $G$ acts trivially on $A_{\Lambda_G}$, the action on $\Lambda_G$ can be extended to $\Gamma$ 
by ${\rm id}$ on $\Lambda_G^\perp\subset \Gamma$.
Note that then $\Gamma^G=\Lambda_G^\perp$ is a non-degenerate lattice of signature $(1,25-\rk\, \Lambda_G)$.
In particular, $\Gamma^G\otimes\RR$ intersects the positive cone $\kc\subset\Gamma\otimes\RR$.
More precisely,  $\Gamma^G\otimes\RR$ intersects
one of the chambers of $\kc$, defined as usual by means of the set of all $(-2)$-classes $\Delta_\Gamma\subset\Gamma$.
Indeed, otherwise there exists one $(-2)$-class $\delta\in \Gamma$ with $\Gamma^G\subset\delta^\perp$
which would yield the contradiction $\delta\in \Lambda_G$.

Next choose an isomorphism  $\Gamma\cong N\oplus U$, where $N$ is the Leech lattice, and
consider a standard generator of $U$ as an isotropic vector $w\in \Gamma$ (the \emph{Weyl vector}). Then the
$(-2)$-classes $\delta$ with $(\delta.w)=1$ are called \emph{Leech roots}. The Weyl group $W\subset\OO(\Gamma)$
is in fact generated by the reflections $s_\delta$ associated with Leech roots (see \cite[Ch.\ 27]{CS}) or, equivalently,
there exists one chamber $\kc_0\subset\kc$ that is described by the condition
$(\delta.\kc_0)>0$ for all Leech roots $\delta$.

Thus, after applying elements of the Weyl group $W$ to the embedding $\Lambda_G\,\hookrightarrow \Gamma$ if necessary,
one can assume that  the distinguished chamber $\kc_0$ is fixed by $G$. Then
$G$ is contained in the subgroup $\Co_\infty\subset\OO(\Gamma)$  of all isometries fixing $\kc_0$. 
The group $\Co_\infty$ is also known to fix the isotropic
vector $w\in\Gamma$ (see \cite{Bor}) and hence $w\in\Gamma^G$. One obtains a
primitive embedding of $\Lambda_G\,\hookrightarrow N$ as the composition
$$\Lambda_G\,\hookrightarrow w^\perp\twoheadrightarrow  w^\perp/\ZZ\cdot w\cong N.$$

Finally, by using $G\subset \Co_\infty\to\Co_0=\OO(N)$ (or by applying Nikulin's general result once more)
one extends the action of $G$ from $\Lambda_G$ to $N$ by the identity on $\Lambda_G^\perp\subset N$.
Then $\Lambda_G^\perp\subset N^G$ (in fact, equality holds)
and, by Lemma \ref{lem:stand}, we ensure $\rk\,N^G\geq4$:

So we proved that  i) implies ii).

\subsection{}\label{sec:ii-i}
For the  converse of Proposition \ref{prop:Gab}, let $G\subset\Co_0$ be a subgroup with $\rk\,N^G\geq4$. One needs to show that
then $G\subset\Aut(\Lambda,\Pi)$ for some positive four space $\Pi\subset\Lambda\otimes\RR$
without $(-2)$-classes contained in $\Pi^\perp$. This is proved in \cite{Gaberdiel} as follows: 

Firstly, one shows the existence of a primitive embedding
\begin{equation}\label{eqn:Cogo}
N_G=(N^G)^\perp\,\hook \Lambda=E_8(-1)^{\oplus 2}\oplus U^{\oplus 4}.
\end{equation}
Such an embedding exists if there exists an orthogonal lattice, i.e.\
an even lattice $M$ with signature $(4,20-\rk\, N_G)$ and discriminant form
$(A_M,q_M)\cong (A_{N_G},-q_{N_G})$, see \cite[Prop.\ 1.5.1]{NikulinInt}. But \cite[Thm.\ 1.12.4]{NikulinInt} implies the existence
of a primitive embedding $N^G\,\hook E_8(-1)\oplus U^{\oplus \rk\, N^G-4}$ and
its orthogonal complement $M$ has the required properties.

Secondly, as $G$ acts trivially on $A_{N_G}\cong A_{N^G}$, its action on $N_G$ can be extended by
${\rm id}$ to $\Lambda$. The orthogonal $N_G^\perp\subset \Lambda$ has signature $(4,20-\rk\,N_G)$
and, therefore, $N_G^\perp\otimes\RR$ contains a positive four space $\Pi$ on which $G$ acts trivially

Thirdly, $N_G$ as a sublattice of the Leech lattice $N$ does not contain any $(-2)$-classes and
for generic choice of $\Pi\subset N_G^\perp\otimes\RR$ neither does $\Pi^\perp$.

\medskip

This concludes the proof of Proposition \ref{prop:Gab}.\qed

\section{Proof of Theorem \ref{thm:main}}\label{sec:Proof}

Combining the previous two sections one can now complete the proof of Theorem \ref{thm:main}.

\subsection{}\label{sec:proofi-ii}
The proof of one direction of Theorem \ref{thm:main} is easy. Indeed, if $G\subset\Aut_s(\Db(X),\sigma)$ for some
$\sigma\in{\rm Stab}^{\rm o}(X)$, then  $G\subset\Aut_s(\widetilde H(X,\ZZ),Z)$ by Proposition \ref{prop:liftHodge}.
If we define $\Pi_{X,Z}$ as the positive four space $$\Pi_{X,Z}:=P_X\oplus P_Z=(H^{2,0}\oplus H^{0,2})(X,\RR)\oplus\RR\cdot{\rm Re}(Z)\oplus \RR\cdot{\rm Im}(Z)$$ and view it as a subspace of $\Lambda\otimes\RR\cong \widetilde H(X,\RR)$, then
$$G\subset\Aut_s(\widetilde H(X,\ZZ),Z)\cong\Aut(\Lambda,\Pi_{X,Z}).$$
Thus, the discussion in Section \ref{sec:i-ii} and Proposition \ref{prop:Gab} apply and show that there exists an 
injection $G\,\hookrightarrow\Co_1$ with invariant lattice of rank at least four. 

\begin{remark}
%
Whenever $\Lambda_G$ can be embedded into a Niemeier lattice
$N_i$ that is not the Leech lattice, then one can argue as in \cite{Kondo} and
deduce the existence of an embedding $G\,\hookrightarrow M_{24}$.
But unfortunately, the Leech lattice cannot be excluded and
one really has to deal with $\Co_1$. Concrete examples have been given
in \cite{Gaberdiel}.
\end{remark}
\subsection{} For the proof of the converse,
an additional problem has to be addressed that was not present in \cite{Gaberdiel}:
One needs to ensure that $\Pi$ in Section \ref{sec:ii-i} can be chosen of the form
$\Pi_{X,Z}$ with $X$ projective and $Z\in H^{1,1}(X,\ZZ)\otimes\CC$.

To achieve this, fix an isomorphism $\Lambda\cong\Lambda_0\oplus U_0$, where we think
of $\Lambda_0=E_8(-1)^{\oplus 2}\oplus U^{\oplus 3}$ as the K3 lattice $H^2(Y,\ZZ)$ and of $U_0\cong
U$ as $(H^0\oplus H^4)(Y,\ZZ)$.
For a subgroup $G\subset\Co_0$ with $\rk\, N^G\geq4$ choose a primitive embedding $N_G\,\hook\Lambda$
as in (\ref{eqn:Cogo}). 


Now, for an arbitrary positive definite primitive sublattice
$L\subset N_G^\perp\subset \Lambda$ of rank four, the lattice 
$L\cap \Lambda_0$, which is the kernel of the projection $L\to U_0$, is of rank at least two.
Hence,  there exists a positive sublattice $L_1\subset L\cap \Lambda_0$ of rank two. We
let $P_1:=L_1\otimes \RR$ be the associated
 positive plane. Due to the surjectivity of the period map, there exists
a K3 surface $X$ with a marking $H^2(X,\ZZ)\cong\Lambda_0$
inducing $(H^{2,0}\oplus H^{0,2})(X,\RR) \cong P_1$. As  the lattice $L_1^\perp\cap\Lambda_0$, which
has signature $(1,19)$, is contained
in $P_1^\perp\subset\Lambda_0\otimes\RR$, there exists a class $\alpha\in H^{1,1}(X,\ZZ)$ with
$(\alpha)^2>0$ and, therefore, $X$ is projective. (In other words, any K3 surface of maximal Picard number $20$ is  projective.)

It remains to find a second positive plane $P_2\subset P_1^\perp\cap(N_G^\perp\otimes \RR)$ 
such that $(P_1\oplus P_2)^\perp$ does not contain any $(-2)$-class. In fact, any such $P_2$
contains elements with non-trivial $H^0$ component and is, therefore, of the form
$P_Z$ for some  $Z=\exp(B+i\alpha)$ with $B,\alpha\in H^{1,1}(X,\ZZ)\otimes\RR$. Thus,
$\Pi:=P_1\oplus P_2$ is of the form $\Pi_{X,Z}$ as required.

In order to show the existence of $P_2$, observe 
that if the intersection of the usual period domain $Q\subset\PP(L_1^\perp\otimes\CC)$
with $\PP((L_1^\perp \cap N_G^\perp)\otimes\CC)$ is contained in the union of all
hyperplanes $\delta^\perp$ orthogonal to some $(-2)$-class $\delta\in L_1^\perp$, then
there exists in fact one $\delta\in L_1^\perp$ orthogonal to $N_G^\perp$.
But this would imply that $N_G$ contains a $(-2)$-class which is absurd for a sublattice
of the Leech lattice.


This concludes the proof of Theorem \ref{thm:main}.\qed

It might be worth pointing out,  that the K3 surfaces constructed in the proof above have all maximal
Picard number $\rho(X)=20$. However, that any group that can be realized at all can also be realized on
one of this type, can also be proved directly. 




%
%
%
%

\section{Symplectic automorphisms of deformations of Hilbert schemes}\label{sec:Aut}

Let $G\subset \Co_0={\rm O}(N)$ be a subgroup with $\rk\,N^G\geq 4$ and choose as before a primitive embedding
$$N_G:=(N^G)^\perp\,\hookrightarrow \Lambda$$ into the extended K3 lattice 
$$\Lambda:=E_8(-1)^{\oplus 2}\oplus U^{\oplus 4}.$$
We also fix a decomposition  $\Lambda =\Lambda_0\oplus U_0$, into the K3 lattice
$\Lambda_0:=E_8(-1)^{\oplus 2}\oplus U^{\oplus 3}$ and a copy $U_0$ of $U$.
By $N_G^\perp\subset \Lambda$ we denote the orthogonal complement of $N_G$ in
$\Lambda$. Then $N_G^\perp$ has
$\rk(N_G^\perp)\geq 4$, more precisely it is a lattice of signature $(4,m)$, and
the action of $G$ on $N_G$ can be extended by ${\rm id}$ on $N_G^\perp$ to an action of $\Lambda$.

In order that a given $G$ can act on a deformation of ${\rm Hilb}^n(X)$ it needs to satisfy an additional
condition:
\smallskip

($\ast$) The lattice $N_G^\perp$ contains a primitive
positive definite lattice $L\subset N_G^\perp$ with $$\ell(L)<\rk(L)=3.$$

\smallskip

We can now state the theorem from the introduction in the following precise form.

\smallskip

{\bf Theorem \ref{thm:main2}.}  \emph{If $G\subset\Co_0$ with $\rk(N^G)\geq 4$
satisfies ($\ast$), then there exist an $n>0$ and a projective irreducible symplectic variety $Y$ deformation equivalent
to ${\rm Hilb}^n(X)$ of a K3 surface $X$ such that $G$ is isomorphic to a subgroup
of the group of symplectic automorphisms ${\rm Aut}_{\rm s}(Y)$ of $Y$.}

\smallskip

\begin{proof}
Consider a primitive, positive definite lattice of rank three $L\subset N_G^\perp$.
Choose bases of $L$ and of its dual $L^*$ such that the natural inclusion
$i\colon L\,\hookrightarrow L^*$ is given by a diagonal matrix ${\rm diag}(a_1,a_2,a_3)$
with $a_i|a_{i+1}$. Then $\ell(L)<3$ if and only of $a_1=\pm1$. 
So, by assumption we can assume that $L\subset N_G^\perp$ contains
a vector $0\ne v\in L$ which is primitive in the overlattice $L\subset L^*$.

Next, we shall apply \cite[Thm.\ 1.14.4]{NikulinInt} twice to $L_1:=v^\perp\cap L$.
The first time, to ensure that there exists a primitive embedding $L_1\,\hookrightarrow \Lambda_0$,
as $\ell(L_1)+2\leq 4\leq\rk(\Lambda_0)-\rk(L_1)=20$, and a second time
to conclude that the induced embedding $L_1\,\hookrightarrow\Lambda_0\,\hookrightarrow\Lambda$
is unique up to ${\rm O}(\Lambda)$. Hence, the given embedding $L_1\subset L\subset N_G^\perp\subset\Lambda$ can be modified by some $\varphi\in{\rm O}(\Lambda)$ such that $\varphi(L_1)\subset\Lambda_0$. By modifying the original embedding of $N_G$ by $\varphi$, we may therefore
assume  that in fact $L_1\subset\Lambda_0$.

\smallskip

Due to the surjectivity of the period map, there exists a K3 surface $X$ and a marking $H^2(X,\ZZ)\cong\Lambda_0$ such that $(H^{2,0}\oplus H^{0,2})(X)\cong L_1\otimes \CC$. Note that $X$ is automatically projective. We denote by $P_1:=L_1\otimes\RR$ the real positive plane associated to $L_1$.
Now choose a generic real positive plane in $P_2\subset P_1^\perp\cap(N_G^\perp\otimes\RR)$
with $v\in P_2$ and such that $(P_1\oplus P_2)^\perp$ does not contain any 
$(-2)$-class. The latter is possible, because otherwise there would be a $(-2)$-class $\delta$
with $(\delta.v')=0$
for any class $v'\in N_G^\perp\otimes \RR$ in an open subset and, therefore, one in $N_G$, which is absurd.
Any such $P_2$ contains elements with non-trivial $H^0$ component and is, therefore,
spanned by the real and the imaginary part of some $Z=\exp(B+i\alpha)$ with
$B,\alpha\in H^{1,1}(X,\RR)$.

 As by construction there are no $(-2)$-classes in $\widetilde H^{1,1}(X,\ZZ)$ orthogonal to $Z$, there exists a stability condition of the form $\sigma=(\kp,Z)\in{\rm Stab}^{\rm o}(X)$.
We may furthermore assume $Z(v)\in\HH\cup \RR_{<0}$. Via the marking
$G$ can be seen a subgroup of the group
${\rm Aut}_s(\widetilde H(X,\ZZ),Z)$ of symplectic Hodge isometries fixing $Z$, for the period $L_1$ and 
the stability function $Z$ are by construction both $G$-invariant.

Due to Theorem \ref{thm:main}, $G$ can thus be realized as a subgroup of ${\rm Aut}_{s}(\Db(X),\sigma)$.

\smallskip

{\bf Claim:} The stability condition $\sigma$ is $v$-generic. 

\smallskip

Suppose $E$ is semistable with $v(E)=v$ and phase $\phi(E)\in(0,\pi]$ and suppose there exists a
semistable subobject $F\,\hookrightarrow E$ in the heart of the stability condition with  $\phi(F)=\phi(E)$.
Decompose $v(F)$ as $v(F)=v_1+v_2$ according to the finite index inclusion
$$L\oplus L^\perp\subset \widetilde H(X,\ZZ),$$ i.e.\ $v_1\in L\otimes\QQ$ and
$v_2\in L^\perp\otimes\QQ$. 
From $v(F)\in P_1^\perp$ and $L^\perp\subset P_1^\perp$, one concludes that $v_1$ is contained in $(L\otimes \QQ)\cap P_2$
which is spanned by $v$. Hence, $v_1=\lambda\cdot v$ for some $\lambda\in\QQ$. Decomposing $v_2$ further
as $v_2=v_2'+v_2''$ with $v_2'\in P_2\cap v^\perp$ and $v_2''\in (P_1\oplus P_2)^\perp$ allows one to write $Z(F)=\lambda\cdot Z(v)+
Z(v_2')$. Now,  $Z(F)\in \RR_{>0}\cdot Z(v)$, as $\phi(F)=\phi(E)$,  and hence $Z(v_2')\in \RR\cdot Z(v)$. 
However, using the injectivity of $Z\colon P_2\,\hookrightarrow \CC$ one finds $v_2'=0$ and, therefore, $Z(F)=\lambda\cdot Z(v)$ with $0< \lambda\leq 1$.
On the other hand, the projection of $v(F)$ under $\Lambda\cong\Lambda^*\to  L^*\subset L\otimes \QQ$ is $v_1=\lambda\cdot v$. As by construction $v$ is primitive in $L^*$, this implies
 $\lambda=1$. But then the semistable quotient
$E/F$  would have Mukai vector $-v_2=-v_2''$ which is annihilated by $Z$. This yields a
contradiction unless $v_2=0$, in which case $F=E$.

\smallskip

Consider the moduli space $M_\sigma(v)$  of $\sigma$-stable objects with Mukai vector $v$ and
phase $\phi\in (0,\pi]$. Then we know by \cite{BaMa} that $M_\sigma(v)$ is a smooth projective
variety birational to a moduli space of stable sheaves on $X$ and therefore, due to \cite{Huy1},
deformation equivalent to it and, eventually, also deformation equivalent to ${\rm Hilb}^n(X)$.

Any $\Phi\in G\subset{\rm Aut}_s(\Db(X),\sigma)$ fixes $\sigma$ and $v$ and, therefore, acts as an automorphism
$$\Phi_v\colon M_\sigma(v)\congpf M_\sigma(v).$$ This yields a homomorphism
$$G\to {\rm Aut}(M_\sigma(v)),\Phi\mapsto \Phi_v.$$

First one observes that all $\Phi_v$ are symplectic. For this it is enough to check that $\Phi_v$ preserves
the natural symplectic structure on $M_\sigma(v)$. This can in fact be verified in one point, say
$[E]\in M_\sigma(v)$. So one has to argue that if $\Phi\colon \Db(X)\congpf\Db(X)$
acts as ${\rm id}$ on $H^{2,0}(X)$, then the isomorphism $\Ext^1(E,E)\cong \Ext^1(\Phi(E),\Phi(E))$
respects the natural pairing given by Serre duality, which is obvious. Second, the map $\Phi\mapsto\Phi_v$ is
injective, as it is compatible with the isomorphism $v^\perp\cong H^2(M_\sigma(v),\ZZ)$. (For simplicity, we assume
here that $M_\sigma(v)$ is fine. Otherwise $M_\sigma(v)$ has to be twisted and the action of $G$ on the
twisted cohomology and hence on the moduli space itself is faithful.)
\end{proof}

\begin{remark}\label{rem:Mong} In \cite{Mongardi},
based on Markman's monodromy operators, Mongardi also finds
a sufficient condition for a group $G\subset \Co_0$ to act as a group of symplectic automorphisms
on a  variety deformation equivalent to a Hilbert scheme. Moreover, he  shows that his condition
is in fact equivalent to ($\ast$). The methods in \cite{Mongardi} should be powerful enough to eventually give a complete
characterization of such groups, whereas it seems unlikely that one can obtain  a necessary condition by our methods.
\end{remark}

\end{document}